\documentclass[a4paper,12pt]{article}
\usepackage{amssymb,amsthm,amsfonts}
\newtheorem{theorem}{Theorem}
\newtheorem*{erd}{\erdos-Rado theorem}

\newtheorem{lemma}[theorem]{Lemma}

\def \N {{{\mathbb N}}}

\newcommand{\Q}{{\mathbb Q}}
\newcommand{\bo}{{\beth_\omega}}
\newcommand{\vs}{{\Q^\bo}}
\newcommand{\erdos}{{Erd\H os}}

\title{Monochromatic Infinite Sumsets}
\author{Imre Leader\footnote{Centre for Mathematical Sciences,
Wilberforce Road,
Cambridge CB3 0WB,
UK,  {\tt I.Leader@dpmms.cam.ac.uk}}\and 
Paul A.~Russell\footnote{Churchill College, Cambridge CB3 0DS, UK,
  {\tt P.A.Russell@dpmms.cam.ac.uk}}}

\begin{document}
\maketitle

\begin{abstract}
We show that there is a rational vector space $V$ such that, whenever $V$ is 
finitely coloured, there is an infinite set $X$ whose sumset $X+X$ is 
monochromatic. 
Our example is the rational vector space of dimension 
$\sup\{\aleph_0,2^{\aleph_0},2^{2^{\aleph_0}},\ldots\,\}$. This complements a
result of
Hindman, 
Leader and Strauss, who showed that the result does not hold for dimension
below $\aleph_\omega$. So our result is best possible under GCH.
\end{abstract}

\begin{section}{Introduction}
It is a well-known consequence of Ramsey's theorem that, whenever the 
naturals are finitely coloured, there is an infinite set $X$ such that all 
pairwise sums of distinct elements of X have the same colour. If one asks 
for a stronger conclusion, that the entire sumset $X+X = \{x+y: x,y \in X\}$ 
is 
monochromatic, then the answer is no: this is because such a sumset automatically 
contains two numbers with one roughly twice the other, and this can easily be ruled
out by a suitable 3-colouring (see e.g.~\cite{h}).

We mention in passing that it is, surprisingly, unknown as to whether or not
this can be achieved with a 2-colouring: this is called Owings' problem
\cite{owings}. For background on this, and other results mentioned in this
introduction, see \cite{hls} -- although we mention that this paper is
self-contained and does not rely on any results from \cite{hls}. 

What happens if one passes to a larger ambient space, for example the 
rationals? Here again, the answer is no: there is a finite colouring of $\Q$ with 
no infinite sumset monochromatic (see e.g.~\cite{hls}). What about for the reals?

Hindman, Leader and Strauss \cite{hls} 
showed that, for every rational vector 
space of dimension smaller that $\aleph_\omega$, there is a finite 
colouring without an infinite monochromatic sumset. Note that this 
establishes the answer for the reals if we assume CH. (It is still unknown 
if the reals have such a bad colouring if we do not make extra 
set-theoretic assumptions.) However, they were unable to find a vector 
space with 
the positive property (of having no bad colourings). 

Our aim in this paper is to show that such a vector space does exist. 
We show that this is the case for any dimension that is at least 
$\beth_\omega$ (read `beth-omega'), which is defined to be 
$\sup\{\aleph_0,2^{\aleph_0},2^{2^{\aleph_0}},\ldots\,\}$.
Note that if we assume GCH then this is 
exactly $\aleph_\omega$, which would be best possible in light of the result 
of \cite{hls}. 
We do not know if the vector space of dimension $\aleph_\omega$ 
has this property if we do not assume GCH.

We also prove a similar result for multiple sums such as $X+X+X$ and so on.
The proof involves a perhaps unexpected use of the Hales-Jewett theorem.

For a finite or infinite cardinal $\kappa$, we write $\Q^\kappa$ to denote
the vector space of dimension $\kappa$ over $\Q$.  That is, $\Q^\kappa$
is the direct \emph{sum} of $\kappa$ copies of $\Q$, not the direct product.
We shall take $\Q^\kappa$ to come equipped with a basis $e_0$, $e_1$,
$e_2$, $\ldots\,$ that is well-ordered by the smallest ordinal of cardinality
$\kappa$.
\end{section}

\begin{section}{Main Result}

Consider $\vs$, the $\bo$-dimensional vector space over $\Q$.  As remarked
above, we shall consider $\vs$ to come equipped with a well-ordered basis
$B$ whose elements we shall denote by $e_0$, $e_1$, $e_2$, $\ldots\,$.

Suppose $x\in\vs$ with $x\ne0$.  We may write $x$ in terms of the basis $B$
and delete all zero entries to obtain a finite list of non-zero rationals.
We call this list the \emph{pattern} of $x$.  More formally, given a non-zero
$x\in\vs$, there is a unique way to express $x$ in the form $x=\sum_{i=1}^n
x_ie_{\alpha_i}$ where $n$ is a positive integer, each $x_i$ is a non-zero
rational and $\alpha_1<\alpha_2<\cdots<\alpha_n$ are ordinals.  The 
\emph{pattern} of $x$ is $(x_1,x_2,\ldots\,,x_n)$.  We shall often
denote the pattern $(x_1,x_2,\ldots\,,x_n)$ simply by $x_1x_2\ldots x_n$.
We say that the pattern $x_1x_2\ldots x_n$ has \emph{length} $n$ and
write $\ell(x_1x_2\ldots x_n)=n$.

Given a finite colouring of $\vs$, we seek an infinite set $X\subset\vs$ with 
$X+X$ monochromatic.  There are two stages to the proof.

We first show (Lemma \ref{patterns}) that, given a 
finite set $\Pi$ of patterns, there is a large
subspace of $\vs$ on which the colour of an $x$ with pattern in $\Pi$ depends
only on the pattern.  The subspace produced is spanned by a subset of the
original basis $B$ of $\vs$.  This part of the proof is a 
fairly standard application of the \erdos-Rado
theorem \cite{er}.

The heart of the proof comes in the second stage.  The main obstacle to
overcome is to determine how we should proceed following the reduction
given by Lemma \ref{patterns}.  That is to say, which patterns should we
consider and how do we force all the elements of $X+X$ to have the desired
pattern or patterns?  While we are able to work within a subspace spanned
by a countable subset $A\subset B$, it is interesting to note that our proof
often requires this subset $A$ to have an order-type greater than $\omega$.
We therefore ask the subspace produced in Lemma \ref{patterns} to have
dimension $\aleph_1$; this allows us to always find $A$ as required.

We now proceed to the first of the two stages detailed above.  First,
we recall the \erdos-Rado theorem. As usual, we denote by $\exp_r(\kappa)$
the $r$-fold exponential of $\kappa$, i.e.~$\exp_0(\kappa)=\kappa$ and
$\exp_{r+1}(\kappa)=2^{\exp_r(\kappa)}$.

\begin{erd}[\cite{er}]
Let $r$ be a non-negative integer and let $\kappa$ be an infinite
cardinal.  Suppose the $(r+1)$-element subsets of a set of 
cardinality $\exp_r(\kappa)^+$ are coloured with $\kappa$ colours.
Then there is a subset of cardinality $\kappa^+$ all of whose
$(r+1)-element$ subsets are the same colour.
\end{erd}

In particular, this immediately implies that for every positive
integer $r$, if the $r$-element subsets of a set of cardinality
$\bo$ are coloured with finitely many colours then there is a subset
of cardinality $\aleph_1$ all of whose $r$-element subsets are the
same colour. 

\begin{lemma}\label{patterns}
Let $k$ be a positive integer
and suppose $\vs$ is $k$-coloured.  Let $\Pi$ be a finite
set of patterns.  Then there is a subset $A\subset B$ of cardinality
$\aleph_1$ such that for each $\pi\in\Pi$ the set
$$\{x\in\vs:x\hbox{\rm\ is in the span of $A$ and has pattern }\pi\}$$
is monochromatic.
\end{lemma}

\begin{proof}
Let $c$ be the given $k$-colouring of $\vs$.

Let $r$ be the length of the longest pattern in $\Pi$.  Let
$$\Pi'=\{\underbrace{00\ldots0}_{r-\ell(\pi)}\pi:\pi\in\Pi\}.$$

Write $\Pi'=\{\pi^{(1)},\pi^{(2)},\ldots\,,\pi^{(n)}\}$.  
We define $n$ $k$-colourings
$c_1$, $c_2$, $\ldots\,$, $c_n$ of the $r$-element subsets of $B$ as follows.
Given $S\subset B$ with $|S|=r$, write $S=\{e_{\alpha_1},e_{\alpha_2},\ldots\,,
e_{\alpha_r}\}$ with $\alpha_1<\alpha_2<\cdots<\alpha_r$.  Then set
$$c_i(S)=\sum_{j=1}^r \pi^{(i)}_je_{\alpha_j}.$$  Now define a single 
$k^n$-colouring $c'$ of the $r$-element subsets of $B$ by
$$c'(S)=(c_1(S),c_2(S),\ldots\,,c_n(S)).$$

We apply the \erdos-Rado theorem to this final colouring $c'$ to obtain
$A'\subset B$ with all $r$-subsets of $A'$ the same colour and 
$|A'|=\aleph_1$.  Removing the $r$ least elements of $A'$,
 we obtain our set $A$ as required.
\end{proof}

We are now ready to proceed to the main part of the proof.

\begin{theorem}\label{main}
Let $k$ be a positive integer, and suppose $\vs$ is $k$-coloured. 
 Then there is an
infinite set $X\subset\vs$ such that the sumset $X+X$ is monochromatic.
\end{theorem}

\begin{proof}
Let $c$ be the given $k$-colouring of $\vs$.

For $a=0$, $1$, $2$, $\ldots\,$, $k$, let $\pi_a$ be the pattern
$$\pi_a=\underbrace{22\ldots2}_{a}\underbrace{11\ldots1}_{2(k-a)}$$
and let $\Pi=\{\pi_a:0\leqslant a\leqslant k\}$.  By Lemma \ref{patterns},
we can find $A\subset B$ with $|A|=\aleph_1$  and colours $c_a$ 
($0\leqslant a\leqslant k$) such that if $x$ is in the span of $A$ and
has pattern $\pi_a$ then $c(x)=c_a$.  By the pigeonhole principle,
we must have $c_a=c_b$ for some $a$ and $b$ with $0\leqslant a<b
\leqslant k$.

Let $C$ be a subset of $A$ of order-type $\alpha=\omega (b-a+2)$ 
and
list the elements of $C$ in order as $f_0$, $f_1$, $f_2$, $\ldots\,$.

Now let $X=\{x_i:i<\omega\}$, where, for each $i<\omega$, we define
$$x_i=\sum_{r=0}^{a-1}f_r+\sum_{r=1}^{b-a}f_{\omega r+i}+
\sum_{r=0}^{2(k-b)-1}\frac{1}{2}f_{\omega (b-a+1)+r}.$$
Then for all $i$, $j\in\N$, we observe that $x_i+x_j$ has pattern
$\pi_a$ or $\pi_b$ according as $i\ne j$ or $i=j$.  Thus $X+X$ is
monochromatic, as claimed.
\end{proof}

\end{section}

\begin{section}{Extensions}
There are two obvious directions in which one might seek to extend
Theorem \ref{main}.

First, what if instead of simply requiring that $X$ be infinite, we
seek an $X$ of cardinality $\aleph_1$, say, or of some larger specified
cardinality?  This is possible if we start with a vector space of sufficiently
large cardinality, and requires only a trivial modification to the proof of
Theorem \ref{main}.

\begin{theorem}
Let $k$ be a positive integer and let $\kappa$ 
be an infinite cardinal.  Then there is an
infinite cardinal $\lambda$ such that whenever the $\lambda$-dimensional
rational vector space $\Q^\lambda$ is $k$-coloured, there is a subset
$X\subset\Q^\lambda$ with $|X|=\kappa$ and $X+X$ monochromatic.
\end{theorem}

Indeed, with a similar application of the \erdos-Rado theorem as above, we may take
$$\lambda=\sup\{\kappa,2^\kappa,2^{2^\kappa},\ldots\,\}.$$

More interestingly, what if rather than simply looking for the sumset
$X+X$ we seek a monochromatic sum of many copies of $X$?  For example,
define the \emph{triple sumset} of $X$ to be 
$$X+X+X=\{x+y+z:x,y,z\in X\}.$$
If we finitely colour $\vs$, can we always find an infinite $X\subset\vs$
with $X+X+X$ monochromatic?

Let us first consider informally how one might
try to extend the proof of Theorem \ref{main} to deal with this problem.
Previously, we split our basis vectors into ``stretches'' of length $\omega$.
Depending on the colouring, we then defined each $x_i$ to either take value
$\frac{1}{2}$ or $1$ on certain fixed coordinates in the stretch
(a ``fixed stretch''),
or we defined each $x_i$ to take value 1 on coordinate $i$ of the stretch and
0 elsewhere (a ``variable stretch'').  
This resulted in $x_i+x_j$ always having a pattern consisting
of 1's and 2's.  More precisely, the pattern on a given fixed stretch is always
the same, whereas the pattern on a variable stretch could be either 
$11$ or $2$.

Now, suppose we consider $x_h+x_i+x_j$ with a similar definition of the
$x_i$.  The variable stretches will now have pattern $111$ or $21$ or $12$
or $3$.  To deal with this, it turns out that we need 
a somewhat unexpected application of the Hales-Jewett Theorem \cite{hj}.

\begin{theorem}
Let $k$ and $t$ be positive integers and suppose $\vs$ is $k$-coloured.  
Then there is an 
infinite set $X\subset\vs$ such that $\underbrace{X+X+\cdots+X}_t$ is
monochromatic.
\end{theorem}

\begin{proof}
Let $c$ be the given $k$-colouring of $\vs$.

Let $\Pi$ be the set of all patterns of 
the form $x_1x_2\ldots x_n$ where $x_1$, $x_2$, $\ldots\,$, $x_n$ are 
positive integers summing to $t$.  Note that $\Pi$ is finite.
Let $N$ be a positive integer such that whenever $\Pi^N$ is $k$-coloured
it contains a monochromatic combinatorial line.  (Such $N$ exists by
the Hales-Jewett Theorem.)  Let $\Pi'$ be the set of patterns obtained
by concatenating $N$ patterns from $\Pi$.

By Lemma \ref{patterns}, there exist a subset $A\subset B$ with $|A|=\aleph_1$ 
and colours $c_\pi$ ($\pi\in\Pi'$) such that if $x$ is
in the span of $A$ and has pattern $\pi$ then $c(x)=c_\pi$.  We induce
a colouring of $\Pi^N$ by giving $(\pi_1,\ldots\,,\pi_N)\in\Pi^N$ the
colour of any $x$ in the span of $A$ with pattern $\pi_1\pi_2\ldots\pi_N$.
(Note that this does not depend on the choice of $x$).

We may now find a monochromatic combinatorial line $L$ in $\Pi^N$.  Let
$J$ be the set of active coordinates of $L$ and, for each $\pi\in\Pi$,
let $I_\pi$ be the set of inactive coordinates where $L$ takes constant
value $\pi$.  (Note that we take our coordinates to range from 0 to $N-1$.)

Let $C$ be a subset of $A$ of order-type $\omega N$ and list the elements
of $C$ in order as $f_0$, $f_1$, $f_2$, $\ldots\,$. Let $X=\{x_i:i<\omega\}$
where
$$x_i=\sum_{r\in J}f_{\omega r+i}+\sum_{\pi\in\Pi}\sum_{r\in I_\pi}
\sum_{s=1}^{\ell(\pi)}\frac{\pi_s}{t}f_{\omega r+s}.$$
Then each element of $\underbrace{X+X+\cdots+X}_t$ has pattern in $L$
and thus $\underbrace{X+X+\cdots+X}_t$ is monochromatic.
\end{proof}

We remark that, exactly as the proof of Theorem 2 was adapted to yield
Theorem 3, we may similarly adapt the proof of Theorem 4 to give:

\begin{theorem}
Let $k$ and $t$ be positive integers and let $\kappa$ be an infinite cardinal.
Then there is an infinite cardinal $\lambda$ such that whenever
$\Q^\lambda$ is $k$-coloured there is an infinite set $X\subset\Q^\lambda$
with $|X|=\kappa$ and $\underbrace{X+X+\cdots+X}_t$ monochromatic.
\end{theorem}

As with Theorem 3, it suffices to take
$$\lambda=\sup\{\kappa,2^\kappa,2^{2^\kappa},\ldots\}.$$
\end{section}


\begin{thebibliography}{99}

\bibitem[1]{er}
Erd\H os, P., and Rado, R., A partition calculus in set theory,
{\it Bull. Amer. Math. Soc} {\bf62} (1956), 427--489.
\bibitem[2]{hj}
Hales, A.~W., and Jewett, R.~I., Regularity and positional games, 
{\it Trans. Amer. Math. Soc.} {\bf106} (1963), 222--239.
\bibitem[3] {h}
Hindman, N., Partitions and sums of integers with repetition, 
{\it J. Comb. Theory (A)} {\bf27} (1979), 19-32.
\bibitem[4]{hls}
Hindman, N., Leader, I., and  Strauss, D., Pairwise sums in colourings of
the reals, {\it Abh. Math. Sem. Univ. Hamburg,} to appear.
\bibitem[5]{owings}
Owings, J., Problem E2494, {\it Amer. Math. Monthly} {\bf81} (1974), 902.

\end{thebibliography}
\end{document}